\documentclass[a4paper]{amsart}
\usepackage{fullpage}
\usepackage[all]{xy}
\usepackage{amsfonts, amsmath, amsthm, amssymb, tikz, accents, bm}
\usepackage[capitalise]{cleveref}
\title{Homological stability for configuration spaces of orbifolds}
\author{Jeffrey Bailes, TriThang Tran}
\date{\today}

\makeatletter
\newtheorem*{rep@theorem}{\rep@title}
\newcommand{\newreptheorem}[2]{%
\newenvironment{rep#1}[1]{%
 \def\rep@title{#2 \ref{##1}}%
 \begin{rep@theorem}}%
 {\end{rep@theorem}}}
\makeatother

\newreptheorem{theorem}{Theorem}
\newreptheorem{lemma}{Lemma}

\newtheorem{theorem}{Theorem}[section]
\newtheorem{lemma}[theorem]{Lemma}

\newtheorem{proposition}[theorem]{Proposition}
\newtheorem{definition}[theorem]{Definition}

\theoremstyle{remark}
\newtheorem{remark}[theorem]{Remark}

\newcommand{\del}{\partial}
\newcommand{\mc}{\mathcal}
\newcommand{\imt}{\mathrm{int}}
\newcommand{\norm}[1]{\lVert#1\rVert}
\newcommand{\st}{\; \vert \;}

\newcommand{\conf}{\mathrm{Conf}}
\newcommand{\pconf}{\mathrm{PConf}}

\newcommand{\sym}{\Sigma}

\newcommand{\stab}{\mathrm{stab}}

\newcommand{\donf}{\mathrm{D}}
\newcommand{\fonf}{\mathrm{C}}

\begin{document}
\maketitle
\begin{abstract}
We prove that homological stability holds for configuration spaces of orbifolds. This builds on the work of Bailes' thesis where he proves that the stabilisation maps are injective.
 \end{abstract}

\section{Introduction}
The goal of this note is to prove the following.

\begin{theorem} \label{thm: main} Let $\mc{X}$ be a connected orbifold of dimension greater than or equal to 2 and let $\conf_n(\mc{X})$ be its configuration space. There are isomorphisms
	\[ H_k(\conf_n(\mc{X}) ; \mathbb{Q}) \cong H_k(\conf_{n+1}(\mc{X}); \mathbb{Q}) \]
for $k \leq n/2$.
\end{theorem}

In \cref{sec: orbifolds}, we will define the configuration space of an orbifold. We first describe the situation for manifolds from which this note has drawn inspiration.

The (unordered) configuration space of a manifold $X$ is
	\[ \conf_n(X) := \{ (x_1, \ldots, x_n) \in (\imt{X})^{\times n} \st x_i \neq x_j \mbox{ for } i \neq j \} / \sym_n, \]
where the symmetric group $\sym_n$ acts on $n$-tuples by $\sigma \cdot (x_1, \ldots, x_n) = (x_{\sigma(1)}, \ldots, x_{\sigma(n)})$. Of particular interest to this paper is a property that these spaces satisfy known as homological stability. Given a connected manifold that is the interior of a manifold with boundary, McDuff \cite{mcduff75} defines a map $\conf_n(X) \to \conf_{n+1}(X)$ by adding a so called point ``at infinity". This map induces a map in homology from $H_k(\conf_n(X))$ to $H_k(\conf_{n+1}(X))$ which is an isomorphism in a range $k \leq n/2$ \cite{mcduff75, segal79}. This phenomenon is known as \emph{homological stability}.

For the configuration space of a closed manifold, one cannot define the stabilisation map of McDuff since there is no point at infinity from which to bring in a new point. Indeed, it can be computed from \cite{fvb62}, that $H_1(\conf_n(S^2);\mathbb{Z}) \cong \mathbb{Z}/(2n-2)\mathbb{Z}$. In particular, the dependence on $n$ means that homological stability does not hold. However, stability results for configuration spaces of closed manifolds with rational coefficients exist (see e.g. \cite{orw13, church12}).

\cref{thm: main} generalises homological stability for configuration spaces of manifolds to configuration spaces of orbifolds.

\subsection{Outline of proof}
The proof of \cref{thm: main} is broken into two parts. We first prove that stability holds in the case when $\mc{X}$ is the interior of an orbifold with boundary, so that we can define a stabilisation map. The proof uses stability for the configuration space of a manifold. The main idea is to decompose $\conf_n(\mc{X})$ into strata with a fixed number of configurations that are on orbifold points of $\mc{X}$. By fixing the number of orbifold points, we notice that the strata behave similarly to that of configurations on a manifold and so we can prove stability for the strata. We can then package everything in the form a spectral sequence in compactly supported cohomology to prove stability for $\conf_n(\mc{X})$. The technique is similar to that used in \cite{kmt15}. 

Lastly, to prove stability for closed manifolds, we employ a standard technique we learnt from \cite{orw13} involving transfer maps.

\subsection{Outline of paper}
In \cref{sec: orbifolds} we define the configuration space of an orbifold. In \cref{sec: filtration} we define a filtration on these configuration spaces. In \cref{sec: maps} we describe the stabilisation maps between configuration spaces. In \cref{sec: stability} we prove homological stability for the case of open orbifolds and in \cref{sec: closed} we prove it for arbitrary orbifolds. We have also included a brief appendix on when certain maps are covering maps.

\subsection*{Acknowledgements}
We would like to thank Dev Sinha and Craig Westerland for many useful discussions relating to this project. We also thank Craig Westerland for reading an earlier draft of this paper.  The first author was supported by the \emph{Albert Shimmins Memorial Fund} during the preparation of this work. The second author was partially supported by the \emph{Alf van der Poorten Travelling Fellowship}. 


\section{Orbifolds} \label{sec: orbifolds}
In this section we define the configuration space of an orbifold. For a more detailed introduction to orbifolds, see \cite{alr07}. Recall that a \emph{groupoid} is a category in which every morphism is an isomorphism. It is common to call the morphisms arrows. Given an orbifold $\mathcal{X}$, we will denote the objects and arrows of $\mc{X}$ by $ob( \mc{X} )$ and $ar(\mc{X})$ respectively. It is also common to use the notation $\mc{X}_0$ for objects and $\mc{X}_1$ for arrows.

A \emph{topological groupoid} is a groupoid such that the objects and the arrows each have a topology. A topological groupoid is \emph{proper} if its $(\text{source},\text{target})$ map: 
	\[ (s,t): ar{ (\mc{X}) } \to ob(\mc{X}) \times ob(\mc{X}) \]
is a proper map. For a topological groupoid $\mathcal{G}$ the \emph{coarse space} of $\mathcal{G}$ is the topological space,
	\begin{align*}
		| \mathcal{G} | := ob(\mathcal{G}) / \sim,
	\end{align*}
	where $x \sim y$ if there is an arrow $x \rightarrow y$ in $ar(\mathcal{G})$.

	An \emph{orbifold structure} on a paracompact Hausdorff space $X$ consists of a topological groupoid $\mathcal{G}$ and a homeomorphism $f : | \mathcal{G} | \rightarrow X$ such that
	\begin{enumerate}
		\item the object space and the arrow space of $\mathcal{G}$ are smooth manifolds.
		\item $\mathcal{G}$ is proper; and
		\item the source and target maps of $\mathcal{G}$ are local diffeomorphisms;
	\end{enumerate}

	An \emph{orbifold}, $\mathcal{X}$, is a topological space, $X$, equipped with an orbifold structure. An \emph{oribifold with boundary} can be defined similarly, where the third condition becomes that the object and arrow spaces are manifolds with boundary. A result of this definition is that $s\circ t^{-1}(\del X) \subset \del X$, where $\del X$ is the boundary of $X$.

In the sequel we will say \emph{let $\mathcal{X}$ be an orbifold} to refer to the structure groupoid ($\mathcal{G}$ above).
The corresponding topological space ($X$ above) can then simply be thought of as being equal to $| \mathcal{X} |$ with homeomorphism $id : | \mathcal{X} | \rightarrow X$. By a \emph{point} on a orbifold, we will mean a point in $ob(\mc{X})$. Given a point $x \in \mc{X}$, a \emph{ghost point} of $x$ is another (distinct) point $y \in \mc{X}$ such that  there is an arrow $(x \to y) \in ar(\mc{X})$.
The \emph{isotropy group} of a point $x \in ob(\mathcal{X})$ is the group $\{ g : x \rightarrow x \, | \, g \in ar(\mathcal{X}) \}$.
We will call a point $x \in ob(\mathcal{X})$ an \emph{orbifold point} if its isotropy group is non-trivial.


The main object of study in this paper is the configuration space of an orbifold.

\begin{definition}
	Let $\mathcal{X}$ be an orbifold (possibly with boundary).
	The \emph{configuration space of $\mathcal{X}$ of $n$-points}, denoted $\conf_n(\mc{X})$ is the orbifold with
	\begin{itemize}
		\item object space
	\[ ob(\conf_n(\mc{X})) := \frac{\{ (x_1, \ldots, x_n) \in \imt (\mc{X}_0)^{\times n} \st \mbox{if } i \neq j \mbox{ then there does not exist } x_i \to x_j \mbox{ in } \mc{X}_1 \}}{ \sym_n}; \]
		\item arrow space being all possible choices of $n$ arrows which take the unordered set of points $\{ x_1, \ldots, x_n \}$ to $\{ x_1', \ldots, x_n' \}$.
			Each $x_i$ must be the source of one of these arrows, each $x_j'$ must be the target of one of these arrows.
	\end{itemize}
\end{definition}

\begin{remark} In the definition, we could also have taken the object space to be the set of ordered $n$-tuples of points in $\mc{X}_0$. Then arrows $(x_1, \ldots, x_n) \to (x_1', \ldots, x_n')$ consist of a permutation $\sigma \in \sym_n$ and arrows $\alpha_i: x_i \to x'_{\sigma(i)} \in \mc{X}_1$. We will make use of this second description in \cref{sec: cover}.
\end{remark}

There is also an ordered version of the configuration space, denoted $\pconf_n(\mc{X})$, which is defined similarly where we do not quotient by the symmetric group action. While the focus of this paper will be on the unordered case, ordered configuration spaces will make an appearance in \cref{sec: closed} when we work with closed orbifolds.

Denote by $B\mc{X}$ the classifying space of $\mc{X}$. By classifying space, we mean the geometric realisation of the nerve of the category $\mc{X}$. We define the (singular) homology of $\mc{X}$, denoted $H_*(\mc{X})$, to be the homology of its classifying space $H_*(B\mc{X})$. We similarly define cohomology and other homotopy invariants of a topological nature for $\mc{X}$. Note that if we are interested in homology with rational coefficients, we have that $H_*(B\mc{X} ; \mathbb{Q}) \cong H_*(|\mc{X}| ; \mathbb{Q})$ (see e.g., Section 1.4 of \cite{alr07}) or Section 4.4 of \cite{moerdijk02}).
We will also make use of cohomology with compact supports. In this instance, we define $H^*_c(\mc{X}; \mathbb{Q}) := H^*_c(|\mc{X}| ; \mathbb{Q})$, (see Chapter 2 of \cite{alr07}) for a more detailed discussion on cohomology with compact supports in terms of De Rham cohomology.
%

\section{A filtration} \label{sec: filtration}
In this section, we will describe a filtration of $\conf_n(\mc{X})$ by the number of orbifold points in a configuration.
Let $\fonf_{n,m}(\mc{X})$ be the suborbispace of $\conf_{n}(\mc{X})$ consisting of configurations with exactly $m$ orbifold points. Collecting these together, we define
 	\[ F_p = \cup_{i=0}^p \fonf_{n,p}(\mc{X}). \]
The suborbispaces $F_p$ form an open increasing filtration of orbifolds such that
	\[ \emptyset = F_{-1} \subset F_0 \subset \cdots F_n = \conf_n(\mc{X}). \]
The suborbispaces $\fonf_{n,p}(\mc{X})$ are therefore given by the filtration differences $F_p - F_{p-1}$.

On the classifying space level, the classifying spaces $BF_p$ form a filtration of $B\conf_n(\mc{X})$ so that the collection over all $p$ of the $BF_p$ give an increasing filtration of $B\conf_n(\mc{X})$ by open subspaces. We will often abuse notation by using $F_p$ to denote both the filtration on classifying spaces and the filtration on the orbifold.

Let $\mc{X}'$ be the orbifold $\mc{X}$ with its orbifold points removed. Therefore $ob(\mc{X}')$ is a manifold that is possibly disconnected. Consider a single component $\mc{X}'^b$ of $\mc{X}$. Let $\donf_{n,m}(\mc{X})$ be the suborbispace of $C_{n,m}(\mc{X})$ where the non-orbifold points of the configuration all lie in $\mc{X}'^b$.

\begin{proposition}\label{prop: fonf is donf} If $\mc{X}$ is a connected orbifold, then the inclusion $i: \donf_{n,m}(\mc{X}) \to \fonf_{n,m}(\mc{X})$ induces an isomorphism on rational homology. That is 
\[i_*: H_*(\donf_{n,m}(\mc{X}; \mathbb{Q})) \cong H_*(\fonf_{n,m}; \mathbb{Q}).\]
\end{proposition}
\begin{proof} 
Consider the diagram

\[ \xymatrix{ 
		\donf_{n,m}(\mc{X}) \ar[r]^i  \ar[drr]_{q'} & \fonf_{n,m}(\mc{X}) \ar[r]^q 	&	|\fonf_{n,m}(\mc{X})| \\
							&						&	|\donf_{n,m}(\mc{X})| \ar[u]_{|i|}
		} \]
where $i$ is the inclusion $\donf_{n,m}(\mc{X}) \to \fonf_{n,m}(\mc{X})$. We have that $q$ and $q'$ are  rational homology equivalences (see e.g., Section 1.4 of \cite{alr07} or Section 4.4 of \cite{moerdijk02}). Moreover, because $\mc{X}$ is connected $|\fonf_{n,m}(\mc{X})|$ is connected and $|\fonf_{n,m}(\mc{X})| = |\donf_{n,m}(\mc{X})|$ and $|i|$ is the identity. The diagram commutes and so $i$ is a rational homology equivalence.
\end{proof}

\section{Stabilisation and transfer maps} \label{sec: maps}
%

%
\subsection{The stabilisation map} Let $\mc{X}$ be a connected orbifold of dimension $\geq 2$ that is the interior of a orbifold $\overline{\mc{X}}$ with boundary. Let $\mathcal{U}$ be the interior of a closed suborbifold of $\partial \overline{\mathcal{X}}$.
Define
\begin{align*}
	\mathcal{V} := \frac{\overline{\mathcal{U}} \times I}{\text{if } u \in ob(\partial \overline{\mathcal{U}}) \text{ and } t, t' \in I \text{ then } (u, t) \sim (u, t')}.
\end{align*}
If we assume that there is an orbifold embedding,
\begin{align*}
	j : \mathcal{V} \rightarrow \overline{\mathcal{X}},
\end{align*}
such that $j|_{\overline{\mathcal{U}} \times \{ 0 \}} : \overline{\mathcal{U}} \hookrightarrow \overline{\mathcal{X}}$ is the standard inclusion, then it is possible to define a stabilisation map. Let $\varepsilon$ be a point on the boundary of $\mc{X}$. We will define a map 
\begin{align*}
	\stab_\varepsilon : Conf_n \mathcal{X} \rightarrow Conf_{n + 1} \mathcal{X}
\end{align*}
that takes a configuration on $\mathcal{X}$ and adds a point to the configuration near $\varepsilon \in \mathcal{U}$. For brevity, we call this map $\stab$.

Using the map $j$, we define a map $k$ which pushes points away from the boundary of $\overline{\mathcal{X}}$, giving a place to add to the configuration.
The map $k$ is defined,
\begin{align*}
		k : \overline{\mathcal{X}} &\rightarrow \overline{\mathcal{X}},\\
	x &\mapsto \left\{
		\begin{array}[]{l l}
			x, & \text{if } x \in \overline{\mathcal{X}} \setminus j(\mathcal{U} \times I),\\
			j(\widehat{k} (j^{-1} (x))), & \text{if } x \in j(\mathcal{U} \times I),
		\end{array}
		\right.
\end{align*}
where,
\begin{align*}
	\widehat{k} : \mathcal{V} = (\overline{\mathcal{U}} \times I) / \sim &\rightarrow \mathcal{V},\\
	(d, t) &\mapsto (d, \tfrac{1}{2} + \tfrac{t}{2}).
\end{align*}
This map is isotopic to the identity. It pushes points into the bottom `half' of the collar around $\mathcal{U}$.
The image under $k$ of any point in $\mathcal{U}$ will end up in the interior of $\overline{\mathcal{X}}$.

Adding a point to a configuration in $Conf_n(\mathcal{X})$ is a two-step process,
\begin{enumerate}
	\item add a point from $\mathcal{U}$ to the configuration, giving $n + 1$ points in $\mathcal{X} \cup \mathcal{U}$; followed by
	\item apply $k$ to push all points into $\mathcal{X}$, giving an element of $Conf_{n + 1} (\mathcal{X})$.
\end{enumerate}
Precisely, for a point $\varepsilon \in \mathcal{U}$ we make the following definition.
\begin{definition} The \emph{stabilisation map} at $\varepsilon$ is the map,
\begin{align*}
	\stab : Conf_n (\mathcal{X}) &\rightarrow Conf_{n + 1} (\mathcal{X}),\\
	\underline{x} &\mapsto k (\underline{x} \cup \{ \varepsilon \}).
\end{align*}
\end{definition}

\subsection{Stabilisation map for compactly supported cohomology}

When working with compactly supported cohomology, we will want a stabilisation map that is an open embedding, since compactly supported cohomology is covariant with respect to such maps. This requires a simple modification of the above map. Note that both stabilisation maps will induce the same map on homology.

Let $\mc{X}$ be the interior of an orbifold $\overline{\mc{X}}$ with boundary. Pick an embedding (as orbifolds) of a $d-1$ dimensional disk $\varphi: D^{d-1} \to \del \overline{\mc{X}}$, where $d$ is the dimension of $\mc{X}$.  

\begin{remark} It is always possible to pick an embedding $\varphi$, if the set of orbifold points on $\del(\mc{X})$ were codimension at least 1 (in $\del(\mc{X})$). In general, this might not be the case. When talking about the stabilisation map in compactly supported cohomology, we will restrict to orbifolds where such an embedding exists. This will not affect the generality of our result as eventually, our proof in the closed case (\cref{thm: stability closed}) will cover all connected orbifolds of dimension greater than 2.
\end{remark}


Fix a diffeomorphism
	\[ \psi : \mathrm{int}(\overline{\mc{X}}\cup_\varphi D^{d-1} \times [0,1)) \to \mc{X}. \]
Define a map
	\[ \stab: \conf_n(\mc{X}) \times \mathbb{R}^n \to \conf_{n+1}(\overline{\mc{X}}\cup_\varphi D^{d-1} \times [0,1)) \to \conf_n(\overline{\mc{X}})\]
by: for $\underline{x} \in \conf_n(\mc{X})$ and $y \in \mathbb{R}^n \cong D^{d-1} \times (0,1)$, send $(\underline{x},y)$ to $(\psi(\underline{x}) , (\psi \circ \varphi) (y))$. Here $\psi(\underline{x})$ is the map the does $\psi$ to each point of $x$.



\subsection{The transfer map} We also want to define the notion of a transfer map which can be intuitively thought of as `all possible liftings to a covering space', but in homology. When dealing with closed orbifolds, we will need to use the transfer map to go between configurations with different numbers of points. This is because the stabilisation map is no longer defined as there is no boundary with which to push in a new point.

We first recall the transfer map for topological coverings. Given a degree $d$ covering of topological spaces, $p : X \rightarrow Y$ there is an induced map in homology, $p_* : H_* X \rightarrow H_* Y$.
The transfer map, however, is a map in the opposite direction,
\begin{align*}
	p^! : H_* Y \rightarrow H_* X.
\end{align*}
The transfer map has the property that $p_* \circ p^!$ is multiplication by $d$.

To apply this idea to a map of orbifolds, $p : \mathcal{X} \rightarrow \mathcal{Y}$, one needs to check that $Bp : B \mathcal{X} \rightarrow B \mathcal{Y}$ is homotopic to a finite sheeted cover. We will want to apply this construction to the configuration space of an orbifold.

For $m < n$, let $\conf_{n,m}(\mc{X})$ be the configuration space of $\mc{X}$ where the $n$ configuration points have been partitioned into subcollections of size $m$ and $n-m$. There is a map
	\[ p : \conf_{n,m}(\mc{X}) \to \conf_n(\mc{X}) \]
obtained by forgetting all the subcollections. In his thesis, Bailes shows that this map is homotopic to a covering map \cite[Chapter 3]{bailes15}. We also summarise the main points in \cref{sec: cover}.
\begin{definition}
Define a collection of maps
	\[ t_{n,m} : H_*(\conf_n(\mc{X})) \to H_*(\conf_{m}(\mc{X})) \]
which are obtained by composing the transfer map $p^!$ with the induced map of $\conf_{n,m}(\mc{X}) \to \conf_m(\mc{X})$ which only remembers the subcollection of $m$ points ($t_{n,m} := forget \circ p^!$). 
\end{definition}
Note that we are slightly abusing language by calling $t_{n,m}$ the \emph{transfer map}. By the transfer map $t$, we will mean $t_{n,n-1}$.

\section{Homological stability} \label{sec: stability}
In this section, we prove that homological stability holds for configuration spaces of open orbifolds. We will use techniques similar to those found in \cite{kmt15}.

In order to prove homological stability, we first prove that stability holds for the filtration differences of the filtration in \cref{sec: filtration}. We then use a spectral sequence in compactly supported cohomology associated to the filtration to conclude that stability holds for the whole space.


Cohomology with compact supports will be another important tool. Observe that if $\conf_n(\mc{X})$ is orientable, then Poincar\'{e} duality for orbifolds gives an isomorphism of the form
	\[ H_*(\conf_n(\mc{X}); \mathbb{Q}) \cong H^{nd-*}_c(\conf_n(\mc{X}); \mathbb{Q})\]
where $\dim(\mc{X}) = d$. Therefore, proving homological stability in the range $* \leq n/2$ is the same as proving \[H^*_c(\conf_n(\mc{X}) ; \mathbb{Q}) \cong H^{*+d}_c(\conf_{n+1}(\mc{X}) ; \mathbb{Q}) \] for $* \geq nd - n/2$.

One caveat for us will be that $\conf_n(\mc{X})$ will not always be orientable. Specifically, $\conf_n(\mc{X})$ is not orientable when $d$ is odd or if $\mc{X}$ is not orientable. In this case, we will need to use a twisted version of Poincar\'{e} duality, involving orientation local systems. Noting that orbifolds are rational homology manifolds, (see e.g., \cite[V.9.2]{bredon97}) we have the following.
\begin{proposition} There is a rational orientation local system $\mc{O}$ on $\conf_n(\mc{X})$ such that:
if $\mc{L}$ is a one dimensional locally constant rational local system, then
	\[ H_*(\conf_n(\mc{X}) ; \mc{L} ) \cong H_c^{n\dim(\mc{X}) - *}(\conf_n(\mc{X}) ; \mc{O} \otimes \mc{L}). \]
Similarly, we have
	\[ H_*(\conf_n(\mc{X}) ; \mc{O} \otimes \mc{L} ) \cong H_c^{n\dim(\mc{X}) - *}(\conf_n(\mc{X}) ; \mc{L}) \]
 by replacing $\mc{L}$ with $\mc{O} \otimes \mc{L}$.
\end{proposition}

We are mainly interested in the case when $\mc{L} = \mathbb{Q}$, the constant one dimensional rational local system. Thus if we want to compute $H_*(\conf_n(\mc{X}) ; \mathbb{Q})$ we can do this by computing $H_c^*(\conf_n(\mc{X}) ; \mc{O})$. We will need the following.

\begin{proposition} \label{prop: trivial coefficients} Let $\mc{O}$ be the orientation local system on $\conf_n(\mc{X})$. Let $i^*\mc{O}$ be the pullback local system induced by the inclusion $i : C_{n,m}(\mc{X}) \to \conf_n(\mc{X})$. Finally let $\mc{O}_n$ be the orientation local system on $C_{n,m}(\mc{X})$.

Then $i^*\mc{O} \otimes \mc{O}_n$ is a trivial local system on $C_{n,m}(\mc{X})$.
\end{proposition}
\begin{proof} 
We first describe the monodromy of the local system $\mc{O}$. Let $\underline{x}$ be a base point in $C_{n,m}(\mc{X})$. Let $d: \pi_1(C_{n,m}(\mc{X}), \underline{x}) \to H_1(\mc{X}, \mathbb{Z})$ be the map that adds the homology classes of the $n$ paths starting and ending at $\underline{x}$. Let $p : \pi_1(C_{n,m}(\mc{X}), \underline{x}) \to \sym_n$ be the map that remembers the permutation of the points of $\underline{x}$  given by a loop. This map requires a choice of ordering of $\underline{x}$. Let $M : H_1(\mc{X}; \mathbb{Z}) \to \mathbb{Z}/2\mathbb{Z}$ be the monodromy associated to the orientation local system of $\mc{X}$. The monodromy of a loop $\gamma$ for the local system $i^*\mc{O}$ can be described by $M(d(\gamma))\epsilon(p(\gamma))$, where $\epsilon(p(\gamma))$ is the sign of the permutation $p$.

The monodromy of $\mc{O}_n$ is similar. In this case, orbifold points and non-orbifold points cannot swap. Let $p_0$ and $p_1$ be defined similarly to $p$, except that they only remember the permutations of the non-orbifold and orbifold points respectively. The monodromy of a loop $\gamma$ in $C_{n,m}(\mc{X})$ associated to the local system $\mc{O}_n$ is given by $M(d(\gamma))\epsilon(p_0(\gamma))\epsilon(p_1(\gamma))$. 

On loops that lie in $C_{n,m}(\mc{X})$, the two monodromies agree since $\epsilon(p(\gamma)) = \epsilon(p_0(\gamma))\epsilon(p_1(\gamma))$. In particular $i^*\mc{O} \otimes \mc{O}_n$ is a trivial local system.
\end{proof}

In order to use compactly supported cohomology, we use the version of the stabilisation map of the form
	\[ \stab: \conf_n(\mc{X}) \times \mathbb{R}^d \to \conf_{n+1}(\mc{X}). \]

This map is an open embedding and so is covariant with respect to compactly supported cohomology. Note that since $\mathbb{R}^d$ is contractible, we have an isomorphism of homology groups $H_*(\conf_n(\mc{X}) \times \mathbb{R}^d) \cong H_*(\conf_n(\mc{X}))$.

The following is analogous to \cref{prop: trivial coefficients} and the proof is essentially the same.

\begin{proposition} \label{prop: trivial coefficients R} Let $\mc{O}'$ be the orientation local system on $\conf_n(\mc{X}) \times \mathbb{R}^d$. Let $(i \times id)^*\mc{O}'$ be the pullback local system induced by the inclusion $i \times id: C_{n,m}(\mc{X}) \times \mathbb{R}^d \to \conf_n(\mc{X}) \times \mathbb{R}^d$. Finally let $\mc{O}_n'$ be the orientation local system on $C_{n,m}(\mc{X}) \times \mathbb{R}^d$.

Then $(i\times id)^*\mc{O}' \otimes \mc{O}_n'$ is a trivial local system on $C_{n,m}(\mc{X}) \times \mathbb{R}^d$.
\end{proposition}

Compactly supported cohomology is functorial with respect to open embeddings so there is an induced map
	\[ \stab: H^*_c(\conf_n(\mc{X}) \times \mathbb{R}^d; \mc{O}') \to H^*_c(\conf_{n+1}(\mc{X}) ; \mc{O}). \]

Similarly, on the filtration differences of the filtration, we have a stabilisation map
	\[ \stab: H_*( \fonf_{n,m}(\mc{X}) \times \mathbb{R}^d ; (i \times id)^*\mc{O}' \otimes \mc{O}_n' ) \to
			H_*( \fonf_{n+1,m}(\mc{X}); i^*\mc{O} \otimes \mc{O}_{n+1} ). \]
			
\begin{proposition} \label{prop: filtration stable} The map
\[	\stab_* : H_*( \fonf_{n,m}(\mc{X}) \times \mathbb{R}^d ; (i \times id)^*\mc{O}' \otimes \mc{O}_n' ) \to
			H_*( \fonf_{n+1,m}(\mc{X}); i^*\mc{O} \otimes \mc{O}_{n+1} )   \]
	is an isomorphism for $* \leq (n-m)/2$
\end{proposition}
\begin{proof} By \cref{prop: trivial coefficients} and \cref{prop: trivial coefficients R} all homology groups have trivial rational coefficients.

Let $X'$ be the orbifold $\mc{X}$ with all orbifold points removed. Let $X^b$ denote the component of $\mc{X'}$ containing the point on the boundary of $\mc{X}$ which we add in by the stabilisation map.
By \cref{prop: fonf is donf} it is enough to prove stability for $\donf_{n,m}(\mc{X})$. Recall that $\donf_{n,m}(\mc{X})$ is the suborbifold of $\fonf_{n,m}(\mc{X})$ such that all non-orbiolfd points are in $X^b$.
There is an obvious diffeomorphism
	\[ \donf_{n,m}(\mc{X}) \cong \conf_{n-m}(X^b) \times \conf_m(\mc{X} - X) \]
	which sends a configuration to its $n-m$ configuration points which are in $X^b$ and its $m$ orbifold points which are in $\mc{X} - X$. The stabilisation map is given by
		\[ \stab \times id : \conf_{n-m}(X^b) \times \conf_m(\mc{X}-X) \to \conf_{n-m+1}(X^b) \times \conf_m(\mc{X}-X). \]
	Now, $X^b$ is a connected open manifold of dimension greater than one, and so by stability for configuration spaces of manifolds (with trivial coefficients) \cite{orw13, segal79}, $\stab_*$ is an isomorphism in the desired range.
\end{proof}

\begin{theorem} \label{thm: stability open} Let $\mc{X}$ be a connected open orbifold that is the interior of an orbifold of dimension $\geq 2$ with boundary such that a boundary component admits a collar.
Then
	\[ \stab : \conf_n(\mc{X}) \to \conf_{n+1}(\mc{X})\] 
induces isomorphisms in rational homology for $* \leq n/2$.
\end{theorem}
\begin{proof}
We will prove that the map is an isomorphism in compactly supported cohomology with twisted coefficients in the range $* \geq nd - n/2$. In particular, we show that

\[ \stab: H^*_c (\conf_n(\mc{X}) \times \mathbb{R}^d ; \mc{O}') \to H^*_c( \conf_{n+1}(\mc{X}) ; \mc{O}) \]
is an isomorphism for $* \geq nd - n/2$.

We use the filtration of $\conf_n(\mc{X})$ in \cref{sec: filtration} with $p$th level of the filtration given by

 \[ F_p = \sqcup_{i=0}^p \fonf_{n,i}(\mc{X}).\] 
 
There is similarly a filtration on $\conf_n(\mc{X}) \times \mathbb{R}^d$ given by $F_p \times \mathbb{R}^d$. Note that the filtration difference $F_{p} - F_{p-1} = C_{n,p}(\mc{X})$ and $(F_p \times \mathbb{R}^d) - (F_{p-1} \times \mathbb{R}^d) = C_{n,p}(\mc{X}) \times \mathbb{R}^d$.

Associated to these filtrations is a spectral sequence in compactly supported cohomology. For details of the construction of this spectral sequence see Section 2 of \cite{kmt15}. In our case, the spectral sequences are of the form
	\[ E^1_{pq} = H^{p+q}_c(\fonf_{n,p}(\mc{X}) \times \mathbb{R}^d ; (i \times id)^*\mc{O}') \]
converging to
	\[ E^{\infty}_{pq} = H^{p+q}_c(\conf_n(\mc{X}) \times \mathbb{R}^d; \mc{O}') \]
and
	\[ {'E}^1_{pq} =  H^{p+q}_c(\fonf_{n,p}(\mc{X}); (i \times id)^*\mc{O} ) \]
converging to
	\[ {'E}^\infty_{pq} = H^{p+q}_c(\conf_n(\mc{X}); \mc{O}). \]

The stabilisation map respects the filtrations so we get a map of spectral sequences 
	\[ \stab: E_{**}^\bullet \to {'E}_{**}^\bullet. \]
By \cref{prop: filtration stable} and twisted Poincar\'{e} duality, the stabilisation map on $E^1_{pq} \to {'E}^1_{pq}$ induces an isomorphism whenever $p+q \geq nd - n/2$. Therefore the map on the $E^\infty$ page of the spectral sequence is an isomorphism for $* \geq nd-n/2$. Poincar\'{e} duality with coefficients in the rational orientation local system then gives the desired result.

\end{proof}

%

\section{Closed orbifolds} \label{sec: closed}
In this section, we prove that homological stability for configuration spaces of orbifolds also holds for closed orbifolds. We follow an argument based on that of \cite{orw13}.

We first define a semisimplicial orbispace associated to the configuration space of an orbifold. Let
	\[ \conf_n(\mc{X})_i := \{ (\underline{x}, p_0, \ldots, p_i) \} \]
where $\underline{x} \in \conf_n(\mc{X})$ and the $\underline{x} \cup (\cup_i \{p_i \}) \in \conf_{n+p+1}(\mc{X})$. Moreover, the $p_i$ should not be orbifold points. The boundary maps $\del_j : \conf_n(\mc{X})_p \to \conf_n(\mc{X})_{p-1}$ are given by forgetting the $j$th $p_i$.

By taking classifying spaces, and the induced boundary maps, we obtain a semisimplicial space whose $i$-simplices are given by $B\conf_n(\mc{X})_i$.

\begin{proposition} \label{prop: fibre} The map
	\[ f: \conf_n(\mc{X})_i \to \mathrm{P}\fonf_{i+1,0}(\mc{X})\]
given by forgetting $\underline{x}$ is a fibre bundle with fibre over $\underline{p} = (p_0, \ldots , p_i)$ given by $\conf_n(\mc{X}_{\underline{p}})$, where $\mc{X}_{\underline{p}}$ is the orbifold $\mc{X}$ with all points with arrows to and from $\underline{p}$ (i.e., $\underline{p}$ and all its ghost points) removed.
\end{proposition}
\begin{remark} $\mathrm{P}\fonf_{i+1,0}(\mc{X})$ is the ordered configuration space of $\mc{X}$ where the configuration points must be non-orbifold points. Because the $p_i$ are non-orbifold points, we have that $(p_0, \ldots, p_i) \in \mathrm{P}\fonf_{i+1,0}(\mc{X})$.
\end{remark}
\begin{proof} If the $p_i$ are not orbifold points, then connectivity of $\mc{X}$ implies that the diffeomorphism type of the orbifold $\mc{X}_{\underline{p}}$ does not depend on $\underline{p}$. That $f$ is a fibre bundle follows for the same reason that $f$ is a fibre bundle for $\mc{X}$ a manifold. 
\end{proof}

A map $\pi:E \to B$ is a \emph{microfibration} if it partially satisfies the homotopy lifting property. That is, if $m\geq 0$ and the following diagram commutes,
\[	\xymatrix{ D^m \times \{0\}	\ar[r] \ar[d]	 &		E \ar[d]^{\pi} \\
			D^m \times [0,1] \ar[r]		&		B
			}
\]
then there exists $\epsilon >0$ such that there is a partial lift $D^m \times [0, \epsilon] \to E$ making the diagram commute. Lemma 2.2 of \cite{weiss05} states that a microfibration with weakly contractible fibres is a weak equivalence.

\begin{lemma} \label{lem: microfibration} The map $\varphi: \norm{ B\conf_n(\mc{X})_\bullet} \to B\conf_n(\mc{X})$ is a weak equivalence.
\end{lemma}
\begin{proof} We will show that $\varphi$ is a microfibration with weakly contractible fibres.

Given the following commutative square
	\[ \xymatrix{ 
				D^n \times \{0\} \ar[r] \ar[d] &	\norm{B\conf_n(\mc{X})}_\bullet \ar[d]^f \\
				D^n \times [0,1] \ar[r]^h & 	B\conf_n(\mc{X})
		} \]
we need to show that there exist $\epsilon > 0$ such that we have a partial lift $\tilde{h} :D^n \times [ 0, \epsilon] \to \norm{B\conf_n(\mc{X})}_\bullet$ that extends the top horizontal map.

We first describe the lift at the level of $\conf_n(\mc{X})$. Now suppose that we have a point $\underline{x} \in \conf_{n}(\mc{X})$. The data of a lift of $\underline{x}$ is a configuration in $\pconf_{i+1}(\mc{X}_{\underline{x}})$ and a simplicial coordinate $t \in \Delta^i$. Note that if $\underline{x}'$ is sufficiently close to $\underline{x}$, then $\underline{x}'$ will be disjoint from the configuration lifting $\underline{x}$, so that this data is also a lift for $\underline{x}'$.

Now, given $(b,s) \in D^m \times [0,1]$ with $s > 0$, define the lift $\hat{\varphi}(b,s)$ to be given by $\varphi(b,s) \in \conf_n(\mc{X})_i$, together with the data of the simplicial coordinate and configurations of $\varphi_0(b,0) \in \norm{\conf_n(\mc{X})_\bullet}$. By the preceding paragraph, for $b \in D^m$,  this is well defined for all $s \leq \epsilon_b$, for some $\epsilon_b > 0$. By compactness, we can find a single $\epsilon >0$ that makes this work, which gives the desired partial lift. By constructing this lift at each level of the nerve and taking classifying spaces, we see that $B\hat{\varphi}$ defines a partial lift of $h$.

We have shown that $\varphi$ is a microfibration.

Let $F(\mc{X}_{\underline{x}})_i$ be the space of $(i+1)$-tuples of distinct non-orbifold points of $\mc{X}_{\underline{x}}$ (distinct in the sense of no arrows). These spaces form a semisimplicial orbispace $F(\mc{X}_{\underline{x}})_\bullet$.
The fibre over $\underline{x}$ is given by the geometric realisation of this semisimplicial orbispace. We show that this semisimplicial orbispace is weakly contractible.

By taking small neighbourhoods of the points in $\underline{x}$, we can find a closed orbifold $\mc{X}' \subset \mc{X}_{\underline{x}}$ which is homotopy equivalent to $\mc{X}_{\underline{x}}$ with some non-oribifold point $y \in (\mc{X}_{\underline{x}}) - \mc{X}'$. 

Now suppose we have map $f : S^k \to \norm{F(\mc{X}_{\underline{x}})_\bullet}$. By the previous homotopy equivalence, we can deform $f$ so that $y$ does not lie in its image. Now, we can fill in $f$ by defining a map $\hat{f} : \mathrm{cone}(S^k) \to \norm{F(\mc{X}_{\underline{x}})_\bullet}$ that sends the cone point to $y$. We can describe $\hat{f}$ using barycentric coordinates. A point in $\mathrm{cone}(S^k)$ is determined by $(s,t)$ with $s \in S^k$ and $t \in [0,1]$ with $(s,1) \sim (s',1)$. A point in $\norm{F(\mc{X}_{\underline{x}})_\bullet)}$ is given by a finite ordered configuration and some barycentric coordinates. We can then define $\hat{f}$ by: if $f(s) = ((p_0, \ldots, p_i), (u_0, \ldots, u_i))$, then $\hat{f}$ sends $(s,t)$ to
	\[ ((p_0, \ldots, p_i, x),((1-t)u_0, \ldots, (1-t)u_i, t)). \]
Therefore $\norm{F(\mc{X}_{\underline{x}})_\bullet}$ is contractible.

By Lemma 2.2 of \cite{weiss05}, a microfibration with contractible fibres is a weak equivalence.
\end{proof}

\begin{proposition} \label{prop: rational inverse} Let $\mc{X}$ be a connected open orbifold of dimension $\geq 2$ admitting a collar. The transfer map $t : H_*(\conf_n(\mc{X}); \mathbb{Q}) \to H_*(\conf_{n-1}(\mc{X}); \mathbb{Q})$ is an isomorphism for $* \leq n/2$.
\end{proposition}
\begin{proof} 
The result follows from Chapter 3 of Bailes' thesis \cite{bailes15} which relies on a result of Dold \cite[Lemma 2]{dold62}.

For brevity let $s_n$ denote the stabilisation map $(\stab)_* : H_*(\conf_n(\mc{X});\mathbb{Q}) \to H_*(\conf_{n+1}(\mc{X});\mathbb{Q})$. Let $t_n$ denote the transfer map $t_n:H_*(\conf_n(\mc{X});\mathbb{Q}) \to H_*(\conf_{n-1}(\mc{X});\mathbb{Q})$.

One can check that the maps $s$ and $t$ satisfy the relations
	\[ t_n \circ s_{n-1} = s_{n-2} \circ t_{n-1} + id. \]
More generally, they satisfy
	\[ t_{n,m}\circ s_{n-1} = s_{m-1} \circ t_{n-1, m-1} + t_{n-1, m}. \]
Furthermore,
	\[ t_{m+1} \circ \cdots \circ t_n = (n-m)!t_{m,n}. \]
This is the contents of Chapter 3 of Bailes' thesis \cite{bailes15}.

Letting $B_n = H_*(\conf_n(\mc{X}); \mathbb{Z})$, $A_n := \mathrm{coker}(s_{n-1}),$ and $\pi_q$ be the projection $B_q \to A_q$ we are now in the situation of Lemma 2 of \cite{dold62}. Dold gives a decomposition of the $B_n$ as $\oplus_{m \leq n} A_n$ and in particular shows the maps $s_n$ are split injective and
	\[ t_{n+1} \circ s_n \]
is multiplication by a nonzero integer constant on each summand. On rational homology, this is an isomorphism so $t_{n+1}$ is an isomorphism whenever $s_n$ is an isomorphism.
\end{proof}

\begin{theorem} \label{thm: stability closed} Let $\mc{X}$ be a connected orbifold of dimension $\geq 2$. The transfer map $t : H_*(\conf_n(\mc{X})) \to H_*(\conf_{n-1}(\mc{X}))$ is an isomorphism for $* \leq n/2$.
\end{theorem}
\begin{proof}
Associated to a semisimplicial space is a spectral sequence that computes the homology of the geometric realisation in terms of the homology of its simplices. Applying the spectral sequence to $\conf_n(\mc{X})_\bullet$ we have a spectral sequence with
	\[ E^1_{pq} = H_q( \conf_n(\mc{X})_p ; \mathbb{Q}) \]
converging to
	\[ E^\infty_{pq} = H_{p+q} (\norm{\conf_n(\mc{X})_\bullet} ; \mathbb{Q}) \]
By \cref{lem: microfibration} the target of the spectral sequence can be identified with $H_* (\conf_n(\mc{X});\mathbb{Q})$.
For each $n$, we therefore get an associated Serre spectral sequence of the form
	\[ E^2_{st} = H_t( \mathrm{P}\fonf_p(\mc{X}) ; H_s (\conf_n(\mc{X}_{\underline{p}}) ; \mathbb{Q}) ), \]
where $\underline{p}$ is a configuration of $p$ non orbifold points, converging to
	\[ E^\infty_{st} = H_{s+t}(\conf_n(\mc{X})_p ; \mathbb{Q}). \]
	
The transfer map respects the simplicial structure of $\conf_n(\mc{X})_\bullet$ and so induces a map of Serre spectral sequences. On coefficients of the $E^2$ page this is given by
	\[ t: H_s(\conf_n(\mc{X}_{\underline{p}});\mathbb{Q} ) \to H_s(\conf_{n-1}(\mc{X}_{\underline{p}}) ; \mathbb{Q}) \]
The orbifolds $\mc{X}_{\underline{p}}$ are open, and admit collars around the points $p$. By \cref{prop: rational inverse}, $t$ induces an isomorphism on the the $E^2$ page (and so the $E^\infty$ page) of the Serre spectral sequence for $s \leq n/2$. But then the transfer map on the $E^1$ page of our original spectral sequence is an isomorphism and so is isomorphism on $E^\infty$ for $* = s+t \leq n/2$. 
\end{proof}

\appendix
\section{Covering maps} \label{sec: cover}
In this appendix, we summarise some of the work in Section 3 of \cite{bailes15}. In particular we want to show the following.

\begin{proposition} \label{prop: cover} The forgetting map
	\[ p : \conf_{n,m}(\mc{X}) \to \conf_n(\mc{X}) \]
is homotopic to a covering map.
\end{proposition}

It is useful to have a desription of the objects and arrows of the orbifold $\conf_{n,m}(\mc{X})$. 

Recall firstly that the orbifold $\conf_n(\mc{X})$ has a description as follows.
\begin{enumerate}
	\item The objects of $\conf_{n}(\mc{X})$ are ordered $n$-tuples, $(x_1, \ldots, x_n)$ of points in $\mc{X}_0$ so that there are no arrows in $\mc{X}_1$ from $x_i \to x_j$ for $i \neq j$.
	\item An arrow of $\conf_n(\mc{X})$ consists of a permutation $\sigma \in S_n$ and an arrow in $\mc{X}_1$ for each point in the configuration. Two elements of the configuration space are connected by an arrow if they differ by a reordering of the points or if a point moves to one of its ghost points. More precisely, there is an arrow $\alpha: \underline{x} \to \underline{y}$ if there exists $\sigma \in S_n$ and $\alpha_1, \ldots, \alpha_n \in \mc{X}_1$ such that
		\[ \alpha_i : x_i \to y_{\sigma(i)} \]
is an arrow for $i =1, \ldots, n$.
\end{enumerate}

With this description, $\conf_{n,m}(\mc{X})$ can be described as the orbifold with
\begin{enumerate}
	\item Objects the same as the objects of $\conf_n(\mc{X})$.
	\item Arrows that have a similar description to the arrows of $\conf_n(\mc{X})$, except that permutations are taken from $\sigma \in S_m \times S_{n-m} \subset S_n$.
\end{enumerate}

The restriction to $S_m \times S_{n-m}$ means that points in a configuration can be grouped into sets of size $m$ and $n-m$.

To show that $p : \conf_{n,m}(\mc{X}) \to \conf_n(\mc{X})$ is homotopic to a covering map, we first recall the definition of a comma category.

\begin{definition} Let $\mc{C}$, $\mc{D}$ be categories and $F: \mc{C} \to \mc{D}$ be a functor. Let $d \in \mc{D}$ be an object. The \emph{comma category} $d \setminus F$ is the category with
	\begin{itemize}
	\item Objects are of the form $(c, f) \in \mathrm{ob}(C) \times \mathrm{mor}(D)$ such that $f \in Hom_{\mc{D}}( d, F(c))$.
	\item Morphisms from $(c_1,f_1)$ to $(c_2,f_2)$ are morphisms $h \in Hom_\mc{C}(c_1, c_2)$ such that $F(h)\circ f_1 = f_2$. 
	\end{itemize}
\end{definition}

We will use the following theorem, which follows from the main theorem of \cite{meyer84}, which is a topologically enriched version of Quillen's Theorem B.

\begin{theorem}\label{thm: mey} Assume that the following condition holds: if $b: \underline{x} \to \underline{x}'$ is an arrow in $\conf_n(\mc{X})$ then $b$ induces a homotopy equivalence
	\[ B( \underline{x} \setminus p) \simeq B( \underline{x}' \setminus p). \]

Then $Bp$ has homotopy fibre $B(\underline{x} \setminus p)$
\end{theorem}

Assuming the conditions of \cref{thm: mey} hold for 
	\[ p: \conf_{n,m}(\mc{X}) \to \conf_n(\mc{X}), \]
we see that 
	\[ hofib(Bp) = B((x_1, \ldots, x_n) \setminus p). \]
	
\begin{proposition} $B( (x_1, \ldots, x_n) \setminus p)$ is homotopically discrete.
\end{proposition}
\begin{proof} 
We first describe a skeletal subcategory of $\underline{x} \setminus p = ((x_1. \ldots, x_n) \setminus p)$. 

Let $\widetilde{S_{n,m}}$ be a set of coset representatives of $S_n / (S_{n-m} \times S_m)$. Consider the category whose objects are
	\[ \{ (x_{\sigma(1)}, \ldots, x_{\sigma(n)}), \sigma) \st \sigma \in \widetilde{S_{n,m}} \} \]
with only identity morphisms. We will show that this category is a skeletal subcategory of $\underline{x} \setminus p$ by showing the following:
\begin{enumerate}
	\item If $(\underline{y}, f)$ is an object in $\underline{x} \setminus p$ then there exists a coset $\sigma \in \widetilde{S_{n,m}}$ such that $(\underline{y},f)$ is connected to $(x_{\sigma(1)}, \ldots, x_{\sigma(n)}), \sigma)$ by an arrow; and
	\item If $\sigma_1, \sigma_2 \in \widetilde{S_{n,m}}$ such that $\sigma_1 \neq \sigma_2$ then there is no arrow in $\underline{x} \setminus p$ from $(x_{\sigma_1(1)}, \ldots, x_{\sigma_1(n)}), \sigma_1)$ to $(x_{\sigma_2(1)}, \ldots, x_{\sigma_2(n)}), \sigma_2)$.
\end{enumerate}

To see the first part, let $(\underline{y}, f)$ be an object in $\underline{x} \setminus p$. We want to find a $\sigma \in \widetilde{S_{n,m}}$ such that there is an arrow $(\underline{y}, f) \to (\sigma(\underline{x}), \sigma)$. By definition, such an arrow $f$, made up of a permutation $\rho \in S_{n-m} \times S_m$ and arrows $\alpha_i \in \mc{X}_1$ such that $\alpha_i : x_i \to y_{\rho^{-1}(i)}$.
We can think of $f$ as a composite $f = \rho \circ \alpha$ where $\rho(\underline{x}')$ means permute the indices of $\underline{x}'$ and $\alpha(\underline{x}')$ means do $\alpha_i$ to the $i$th entry of $\underline{x}$.

Now $\alpha$ and $\rho$ give an arrow in $\conf_{n,m}(\mc{X})$ by $\rho \circ \alpha^{-1}$ which one can check defines a map in the comma category by
	\[ ((y_1, \ldots, y_n), f) \mapsto ((x_{\rho(1)}, \ldots, x_{\rho(n)}), \rho). \]
	
Now since $\widetilde{S_{n,m}}$ is the set of coset representatives of $S_n / (S_{n-m} \times S_m)$, there exists a $\sigma \in \widetilde{S_{n,m}}$ such that $\rho \in [\sigma]$. Therefore $\sigma \circ \rho^{-1} \in S_{n-m} \times S_m$. Then $\sigma \circ \rho^{-1}$ is an arrow in $\conf_{n,m}(\mc{X})$ which defines an arrow in $\underline{x} \setminus p$,
	\[ \sigma \circ \rho^{-1}: ((x_{\rho(1)}, \ldots, x_{\rho(n)}), \rho) \to ((x_{\sigma(1)}, \ldots, x_{\sigma(n)},), \sigma). \]

Composing the two arrows we have constructed, we get an arrow
	\[ (\sigma \circ \rho^{-1}) \circ (\rho \circ \alpha^{-1}) : ((y_1, \ldots, y_n) ,f) \to ((x_{\sigma(1)}, \ldots, x_{\sigma(n)}), \sigma), \]
where $\sigma \in \widetilde{S_{n,m}}$. This completes the first part.

For the second part, we need to show that if $\sigma_1, \sigma_2 \in \widetilde{S_{n,m}}$ such that $\sigma_1 \neq \sigma_2$ then there is no arrow of the from $(\sigma_1(\underline{x}), \sigma_1) \to (\sigma_2(\underline{x}), \sigma_2)$.

Suppose for contradiction that such an arrow exists. Then there is an arrow in $\conf_{n,m}(\mc{X})$ of the form
	\[ (x_{\sigma_1(1)}, \ldots, x_{\sigma_1(n)}) \to (x_{\sigma_2(1)}, \ldots, x_{\sigma_2(n)}). \]
The arrows in $\conf_{n,m}(\mc{X})$ are formed from a re-ordering $\rho \in S_{n-m} \times S_m$ and $n$ arrows $\alpha_i : x_{\sigma_1(i)} \to x_{\rho \circ \sigma_2(i)}$ in $\mc{X}_1$. Since there are no arrows $x_i \to x_j$, it must be that $s(\alpha_i) = t(\alpha_i)$, i.e., for each $i$, $\alpha_i$ has the same source and target. That is
	\[ x_{\sigma_1(i)} = x_{\rho \circ \sigma_2(i)}. \]
Therefore $\rho \circ \sigma_1 = \sigma_2$, so $[\sigma_1] = [\sigma_2]$ which is a contradiction.

We have now shown that our subcategory of $\underline{x} \setminus p$ is skeletal. Note that it is discrete and so $\underline{x} \setminus p$ is homotopically discrete.
\end{proof}

We now want to check that the conditions of \cref{thm: mey} are satisfied.
\begin{proposition} If $b: \underline{x} \to \underline{x}'$ is an arrow in $\conf_n(\mc{X})$ then $b$ induces a homotopy equivalence
	\[ B( \underline{x} \setminus p) \simeq B( \underline{x}' \setminus p). \]
\end{proposition}
\begin{proof} Let $b^* : \underline{x} \setminus p \to \underline{x}' \setminus p$ be the map induced by $b$. On objects it is given by
	\[ b^*_0(c,f) = (c, f \circ b).\]
On morphisms acts as the identity on $h \in (\conf_n(\mc{X}))_0$ since the following diagram commutes.
\[ \xymatrix{
	&	\underline{x}' \ar[d]^b		\ar[ddl]_{f\circ b}	\ar[ddr]^{g \circ b}	&	\\
	&	\underline{x} \ar[dl]^f	\ar[dr]_g	&	\\
p(\underline{y}) \ar[rr]_{p(h) = p(b^*(h))}	&	&	p(\underline{z}) 
} \]

We want to show that $b^*$ is an isomorphism of categories. If $b = \sigma \circ \alpha$, then define $b^{-1} = \alpha^{-1} \circ \sigma^{-1} = \sigma^{-1} \circ \gamma$, where $\gamma = \sigma(\alpha^{-1})$. One can then check that $(b^{-1})^*$ is an inverse to $b^*$. A careful check of this can be found in Section 3 of \cite{bailes15}.
\end{proof}

Applying \cref{thm: mey} we have now shown \cref{prop: cover}.

\bibliographystyle{alpha}
\bibliography{references}

\begin{thebibliography}{KMT15}

\bibitem[ALR07]{alr07}
Alejandro Adem, Johann Leida, and Yongbin Ruan.
\newblock {\em Orbifolds and stringy topology}, volume 171 of {\em Cambridge
  Tracts in Mathematics}.
\newblock Cambridge University Press, Cambridge, 2007.

\bibitem[Bai15]{bailes15}
Jeffrey Bailes.
\newblock Orbispaces, configurations and quasi-fibrations.
\newblock 2015.
\newblock PhD Thesis, University of Melbourne.
  http://hdl.handle.net/11343/57345.

\bibitem[Bre97]{bredon97}
Glen~E. Bredon.
\newblock {\em Sheaf theory}, volume 170 of {\em Graduate Texts in
  Mathematics}.
\newblock Springer-Verlag, New York, second edition, 1997.

\bibitem[Chu12]{church12}
Thomas Church.
\newblock Homological stability for configuration spaces of manifolds.
\newblock {\em Invent. Math.}, 188(2):465--504, 2012.

\bibitem[Dol62]{dold62}
Albrecht Dold.
\newblock Decomposition theorems for {$S(n)$}-complexes.
\newblock {\em Ann. of Math. (2)}, 75:8--16, 1962.

\bibitem[FVB62]{fvb62}
Edward Fadell and James Van~Buskirk.
\newblock The braid groups of {$E\sp{2}$} and {$S\sp{2}$}.
\newblock {\em Duke Math. J.}, 29:243--257, 1962.

\bibitem[KMT15]{kmt15}
Alexander Kupers, Jeremy Miller, and TriThang Tran.
\newblock Homological stability for symmetric complements.
\newblock {\em Trans. Amer. Math. Soc}, Electronically published on Dec 2 2015.
\newblock (to appear in print).

\bibitem[McD75]{mcduff75}
Dusa McDuff.
\newblock Configuration spaces of positive and negative particles.
\newblock {\em Topology}, 14:91--107, 1975.

\bibitem[Mey84]{meyer84}
Jean-Pierre Meyer.
\newblock Mappings of bar constructions.
\newblock {\em Israel J. Math.}, 48(4):331--339, 1984.

\bibitem[Moe02]{moerdijk02}
Ieke Moerdijk.
\newblock Orbifolds as groupoids: an introduction.
\newblock In {\em Orbifolds in mathematics and physics ({M}adison, {WI},
  2001)}, volume 310 of {\em Contemp. Math.}, pages 205--222. Amer. Math. Soc.,
  Providence, RI, 2002.

\bibitem[RW13]{orw13}
Oscar Randal-Williams.
\newblock Homological stability for unordered configuration spaces.
\newblock {\em Q. J. Math.}, 64(1):303--326, 2013.

\bibitem[Seg79]{segal79}
Graeme Segal.
\newblock The topology of spaces of rational functions.
\newblock {\em Acta Math.}, 143(1-2):39--72, 1979.

\bibitem[Wei05]{weiss05}
Michael Weiss.
\newblock What does the classifying space of a category classify?
\newblock {\em Homology Homotopy Appl.}, 7(1):185--195, 2005.

\end{thebibliography}

\end{document}